\documentclass[12pt]{amsart}
\usepackage[utf8]{inputenc}
\usepackage{amsmath,amsthm,amscd,amssymb}
\usepackage[margin=1in]{geometry}
\usepackage{color}
\usepackage[colorlinks=true, linkcolor=blue, citecolor=blue, urlcolor=blue]{hyperref}

\theoremstyle{plain}

\theoremstyle{definition}

\theoremstyle{remark}

\theoremstyle{plain}
\newtheorem{theorem}{Theorem}[section]
\newtheorem{lemma}[theorem]{Lemma}

\theoremstyle{definition}

\newtheorem{remark}[theorem]{Remark}

\newcommand{\bR}{\mathbb{R}}
\newcommand{\sq}{\square}

\begin{document}

\title{A Differential Harnack Inequality for the FitzHugh-Nagumo Equation}

\author[Cao]{Xiaodong Cao}
\address{Department of Mathematics, Cornell University \\ Ithaca, NY 14853, USA}
\email{xiaodongcao@cornell.edu}

\author[Fluck]{Harry Fluck}
\address{Department of Mathematics, Cornell University \\ Ithaca, NY 14853, USA}
\email{hpf5@cornell.edu}

\author[Teng]{Sophia Teng}
\address{ Cornell University \\ Ithaca, NY 14853, USA}
\email{steng1129@gmail.com}

\date{\today}
\begin{abstract}
In this paper we develop a Li-Yau-Hamilton (LYH) type differential Harnack estimate for positive solutions to the FitzHugh-Nagumo equation on $\bR^n$. We then use our LYH-differential Harnack inequality to prove several properties about positive solutions to the equation, including a classical Harnack inequality and a lower bound for the speed of traveling wave solutions. 
\end{abstract}
\maketitle

\section{Introduction}
Differential Harnack inequalities have become one of the most powerful tools in the study of parabolic partial differential equations. Beginning with the pioneering work of P. Li and S.-T. Yau \cite{ly86} on the heat equation and R. Hamilton's subsequent work on geometric evolution equations \cite{hamilton93}, such inequalities provide pointwise estimates involving both the solution and its derivatives. By integrating along space-time curves, differential Harnack inequalities yield classical Harnack inequalities comparing solutions at different points in space and time. They have also found numerous applications in the study of geometric flows, such as gradient estimates, monotonicity formulas, and Liouville-type results.\\

In this paper, we study positive smooth solutions
$f:\mathbb R^n\times [0,\infty)\rightarrow\mathbb R$ of the scalar
FitzHugh--Nagumo equation
\begin{equation}\label{eq:1}
f_t=\Delta f-f(1-f)(a-f),
\qquad -1< a<0 .
\end{equation}

The FitzHugh–Nagumo model was originally introduced by FitzHugh \cite{FitzHugh1961} and Nagumo \cite{Nagumo1962} as a simplification of the Hodgkin–Huxley equations \cite{HodgkinHuxley1952} to describe the propagation of electrical impulses in nerve cells. Since then it has become one of the fundamental mathematical models for excitable media, with applications ranging from neuroscience and cardiac electrophysiology to combustion theory, chemical reactions, and pattern formation \cite{KeenerSneyd,Murray2002}. Under suitable assumptions, the full FitzHugh–Nagumo system reduces to the above scalar bistable reaction–diffusion equation, which is commonly referred to as the FitzHugh–Nagumo equation \cite{KeenerSneyd}. The parameter $a$ determines the bistable structure of the nonlinearity. In the limiting case $a=-1$, the equation
\eqref{eq:1} reduces to the Newell-Whitehead-Segel equation (\cite{nw69, segel})
\[
f_t=\Delta f+f(1-f)(1+f),
\]
which appears in the study of pattern formation and nonlinear wave
propagation. Differential Harnack inequalities for the Newell–Whitehead–Segel equation were previously established by Booth et al \cite{bbc19}, corresponding to the endpoint case $a=-1$. The goal of the present paper is to extend this result to the remaining bistable regime $-1<a<0$.

The proof requires several new ingredients compared with the endpoint case $a=-1$. In particular, the additional parameter $a$ introduces new nonlinear terms into the evolution of the Harnack quantity, requiring a refined choice of auxiliary constants and a more delicate maximum principle argument. As applications of the differential Harnack inequality, we derive a classical Harnack inequality and obtain quantitative estimates for traveling wave solutions.

The mathematical theory of the FitzHugh-Nagumo equation has been extensively developed over the past several decades. Topics that have received considerable attention include the existence and uniqueness of solutions, qualitative properties, asymptotic behavior, stability, bifurcation theory, and singular perturbation analysis. Of particular interest are traveling wave solutions, which describe the propagation of excitation fronts in excitable media. Their existence, stability, and geometric structure have been investigated using phase-plane analysis, geometric singular perturbation theory, and related techniques; see, for example, \cite{LiuVanVleck2005,VolpertVolpertVolpert1994,Smoller1994}. Explicit traveling wave solutions have also been constructed using the first integral method and related approaches \cite{lg06}.

Despite this extensive literature, differential Harnack inequalities for the FitzHugh–Nagumo equation appear to be absent. Since differential Harnack estimates have proved to be fundamental tools for understanding nonlinear parabolic equations, it is natural to ask whether analogous estimates exist for the FitzHugh–Nagumo equation and whether they can be used to derive new qualitative information about its solutions.\\

Our main goal is to establish a Li--Yau--Hamilton type differential
Harnack inequality for positive solutions of \eqref{eq:1}. Let
\[
u=\log f
\]
and define
\[
F(u)=(1-e^u)(e^u-a).
\]
We consider the Harnack quantity
\[
H=\alpha\Delta u+\beta|\nabla u|^2+cF(u)+\varphi(t),
\]
where $\alpha,\beta,c$ are constants satisfying appropriate conditions and
$\varphi(t)$ is an explicitly chosen time-dependent correction term.
Our main theorem proves that
\[
H\geq 0
\]
for all positive solutions of the FitzHugh--Nagumo equation.

\renewcommand{\theenumi}{\roman{enumi}}

\begin{theorem}
\label{maintheorem}
Let $f > 0$ be a smooth solution to \eqref{eq:1}, define $u = \log f$ and
\begin{align*}
    F(u) &= (1 - e^u)(e^u - a).
\end{align*}
Then, for every $(x,t)\in\mathbb{R}^n\times(0,\infty)$, we have
\begin{equation}\label{H}
H (x,t)=\alpha \Delta u + \beta |\nabla u|^2 + c F(u) + \varphi(t) \geq 0,
\end{equation}
 provided that:
\begin{enumerate}
    \item $ \alpha >\beta >0$; \label{iden:main_thm_1}
    \item  $\frac{n\alpha^{2}}{2(\alpha-\beta)}>c> \frac{n\alpha^{2}(2\alpha+\beta) }{3n\alpha^{2}-2\beta(\alpha-\beta)}$, $c\neq \frac{n\alpha^{2}}{4(\alpha-\beta)}$; \label{iden:main_thm_2}
\end{enumerate}
and \[\varphi(t)=\frac{\eta}{1-e^{2\eta\mu t}}(\frac{1}{\nu-\mu}e^{2\mu\eta t}-\frac{1} {\nu+\mu})+A,\]
where $ A,\mu,\nu, \eta>0 $ are explicit positive constants depending on $\alpha,\beta, a, c$ and the dimension $n$.
\end{theorem}

\begin{remark}It is straightforward to verify that the interval in \eqref{iden:main_thm_2} is nonempty whenever $n\geq 2$. In dimension $n=1$, the interval is nonempty provided $2\beta > \alpha$. Also, $A>0$  under the conditions that $a> -1$ and $c\neq \frac{n\alpha^{2}}{4(\alpha-\beta)}$. We refer the reader to the  proof of Theorem \ref{maintheorem} (Section \ref{sec:main_proof}) for the precise definition of the above constants.
\end{remark}


\begin{remark}The differential Harnack inequality obtained here is unlikely to be sharp. For several classical parabolic equations, including the heat equation and the Ricci flow, sharp differential Harnack inequalities are characterized by distinguished self-similar solutions. It would therefore be interesting to determine whether the Harnack quantity introduced here can be improved to a sharp estimate, whether equality characterizes traveling wave solutions, and whether analogous differential Harnack inequalities hold for the full FitzHugh–Nagumo system or for more general reaction–diffusion equations.
\end{remark}

The proof follows the maximum principle approach introduced by Hamilton.
The main difficulty comes from controlling the nonlinear reaction terms
which arise in the evolution equation of the Harnack quantity. By choosing
the constants appropriately and introducing a suitable time-dependent
function $\varphi(t)$, we are able to control these additional terms and
obtain a differential Harnack estimate.

As a first application, we integrate the differential Harnack inequality
along space-time curves to obtain a classical Harnack inequality. This gives
a comparison between the values of a positive solution at different points
in space and time, extending the classical Li-Yau estimate to this nonlinear
reaction-diffusion setting.

\begin{theorem}\label{thm:app}
Let $f$ be a positive solution to \eqref{eq:1}, $u=\operatorname{log}(f)$, $\alpha, \beta ,c, \mu, \eta ,\nu$ be as in Theorem \ref{maintheorem}, and $(x_1,t_1), (x_2,t_2)\in \mathbb{R}^n\times (0,\infty)$ with $t_2> t_1$. If $0< f\leq 1$ we have
\begin{align*}
    \frac{f(x_2,t_2)}{f(x_1,t_1)}\geq   \operatorname{exp}\{-\frac{1}{4(1-\frac{\beta}{\alpha})}\frac{|x_{2}-x_{1}|^{2}}{t_{2}-t_{1}}-\frac{(t_2-t_1)}{4}(\frac{c}{\alpha}-1)_{+}(1-a)^2 -\frac{A(t_{2}-t_{1})}{\alpha}\big\} \nonumber \\ \hspace{17.5mm} \times \big[\frac{1-e^{-2\mu\eta t_{2}}}{1-e^{-2\mu\eta t_{1}}}\big]^{-\frac{1}{2\mu\alpha(\nu+\mu)}}\times \big[ \frac{e^{2\mu\eta t_{2}}-1}{e^{2\mu\eta t_{1}}-1}]^{\frac{1}{2\alpha\mu(\nu-\mu)}}.
\end{align*}
If $f\geq 1$ and $c\geq \alpha $
\begin{align*}
    \frac{f(x_2,t_2)}{f(x_1,t_1)}\geq  \operatorname{exp}\{-\frac{1}{4(1-\frac{\beta}{\alpha})}\frac{|x_{2}-x_{1}|^{2}}{t_{2}-t_{1}} -\frac{A(t_{2}-t_{1})}{\alpha}\big\} \nonumber \\ \hspace{17.5mm} \times \big[\frac{1-e^{-2\mu\eta t_{2}}}{1-e^{-2\mu\eta t_{1}}}\big]^{-\frac{1}{2\mu\alpha(\nu+\mu)}}\times \big[ \frac{e^{2\mu\eta t_{2}}-1}{e^{2\mu\eta t_{1}}-1}]^{\frac{1}{2\alpha\mu(\nu-\mu)}}.
\end{align*}
Here $r_{+}=\max\{r,0\}$.

\end{theorem}
\begin{remark}
    A standard maximum principle argument shows that if $f$ satisfies either $0< f\leq 1$ or $f\geq 1$ at $t=0$, then it remains so for all time.
\end{remark}

We also apply our estimate to traveling wave solutions of
\eqref{eq:1}. In particular, a positive solution $f>0$ to \eqref{eq:1} is called a traveling wave solution if $f(x,t)=v(x_1,...,x_{n-1},x_n+st)$ for some function $v:\mathbb{R}^n\to \mathbb{R}$ and $s\in \mathbb{R}$. Rewriting the Harnack inequality in terms of the original
solution $f$, we derive constraints on the wave speed $s$. The differential Harnack estimate then provides a lower
bound for the wave speed under suitable assumptions on the limiting behavior
of the traveling wave.

\begin{theorem}\label{thm:app_2}
    For $n\geq 2$, let $f(x,t)=v(x_1,...,x_{n-1},x_n+st)>0$ be a traveling wave solution of \eqref{eq:1} and suppose that $\lim_{t\to\infty} f(x, t)=0$ for some $x\in \mathbb{R}^n$. Let $\alpha,\beta,c,\mu,\eta,\nu$, and $\varphi$, as in Theorem \ref{maintheorem}, then we have
    \begin{align}\label{eq:speed_bound}
        s^2\geq \frac{4(\alpha-\beta)}{\alpha^2}((c-\alpha)a-\varphi_\infty),
    \end{align}
    where
\[
\gamma=\frac{2(\alpha-\beta)}{n\alpha^2}
      =\mu^2-\nu^2
\]
and
\[
\varphi_\infty
:=\lim_{t\to\infty}\varphi(t)
=A+\frac{\eta(\nu+\mu)}{\gamma}.
\]  In particular, whenever
    \begin{align}\label{eq:speend_bound2}
        (c-\alpha)a>\varphi_{\infty},
    \end{align}
    \eqref{eq:speed_bound} gives a non-trivial lower bound for $s$.
\end{theorem}

\begin{remark}
    For $n\geq 2$, one can verify that there exist choices of $\alpha, \beta, $ and $c$ satisfying both the hypotheses of Theorem \ref{maintheorem} and the condition \eqref{eq:speend_bound2}. Therefore, the lower bound in Theorem~\ref{thm:app_2} is
nontrivial for suitable choices of parameters.
\end{remark}


The remainder of the paper is organized as follows. In Section \ref{sec:tech}, we prove several technical lemmas that will be needed in the proof of Theorem \ref{maintheorem}. In Section \ref{sec:main_proof}, we prove Theorem \ref{maintheorem}. In Section \ref{sec:applications}, we apply Theorem  \ref{maintheorem} to prove Theorem \ref{thm:app} and Theorem \ref{thm:app_2}.



\section{Technical Lemmas}\label{sec:tech}

In this section, we prove some technical lemmas that will be needed throughout the proof of Theorem \ref{maintheorem}. 
For notational convenience, we introduce the box operator $\sq g(x,t) := g_t - \Delta g$. 

First, we calculate the evolution of the Harnack quantity. 
\begin{lemma} \label{lem:computations}
Let $f > 0$ be a solution to \eqref{eq:1}, $u = \log f$, and
\begin{align*}
    F(u) &= (1 - e^u)(e^u - a).
\end{align*}
For $\alpha, \beta, c\in \mathbb{R}$ and $\varphi(x,t)\in C^\infty(\mathbb{R}^n\times\mathbb {R}_+)$, set
\begin{equation}\label{H}
H =\alpha \Delta u + \beta |\nabla u|^2 + c F(u) + \varphi.
\end{equation}
Then
\begin{align}
\sq H =& 2\nabla u \cdot \nabla H + 2(\alpha-\beta) |\nabla \nabla u|^2 + \varphi_t - \Delta \varphi-2\nabla u \cdot \nabla \varphi  \\
& + \alpha (1 + a - 2e^u) e^u \cdot (\Delta u) + e^u \cdot |\nabla u|^2 (1 + a)(\alpha + 2 \beta - 2c) \nonumber \\
& + e^u \cdot |\nabla u|^2 (- 2 e^u)(2 \alpha + 2 \beta - 3c) + c(1+a-2e^u) e^u (1-e^u)(e^u - a) \nonumber \\ 
=& 2\nabla u \cdot \nabla H + 2(\alpha-\beta) |\nabla \nabla u|^2 + \varphi_t - \Delta \varphi-2\nabla u \cdot \nabla \varphi 
\\
& + (1+a - 2e^u) e^u \cdot (H-\beta |\nabla u|^{2}-cF-\varphi) + e^u \cdot |\nabla u|^2 (1 + a)(\alpha + 2 \beta - 2c) \nonumber \\
& + e^u \cdot |\nabla u|^2 (- 2 e^u)(2 \alpha + 2 \beta - 3c) + c(1+a-2e^u) e^u (1-e^u)(e^u - a) \nonumber. 
\end{align}
\end{lemma}

\begin{proof}
Note that we have
\begin{align*}
    \partial_{t} u=\frac{\partial_{t}f}{f}&=\frac{\Delta f}{f}-a+(1+a)f-f^{2}\\&= \Delta u+|\nabla u|^{2}-a+(1+a)e^{u}-e^{2u}.
\end{align*}
Using this, we get that:
\begin{align*}
   \Box \Delta u&= \Delta (\partial_{t}u)-\Delta\Delta u\\&= \Delta( \Delta u+|\nabla u|^{2}-a+(1+a)e^{u}-e^{2u})-\Delta\Delta u\\&= 2\Delta\nabla u\nabla u+2|\nabla\nabla u|^{2}+(1+a)e^{u}(|\nabla u|^{2}+\Delta u)-4e^{2u}|\nabla u|^{2}-2e^{2u}\Delta u\\&=2\nabla\Delta u\nabla u+2|\nabla\nabla u|^{2}+(1+a)e^{u}(|\nabla u|^{2}+\Delta u)-4e^{2u}|\nabla u|^{2}-2e^{2u}\Delta u
\end{align*}
\begin{align*}
    \Box |\nabla u|^{2}&=2\nabla(\partial_{t}u)\nabla u-\Delta|\nabla u|^{2}\\&=2\nabla(\Delta u+|\nabla u|^{2}-a+(1+a)e^{u}-e^{2u})\nabla u-2\Delta \nabla u.\nabla u-2|\nabla\nabla u|^{2}\\&=2\nabla(|\nabla u|^{2}-a+(1+a)e^{u}-e^{2u})\nabla u-2|\nabla\nabla u|^{2}\\&=2(\nabla|\nabla u|^{2}.\nabla u+(1+a)e^{u}|\nabla u|^{2}-2e^{2u}|\nabla u|^{2}-|\nabla \nabla u|^{2})
\end{align*}
\begin{align*}
    \Box F&= (1+a-2e^{u})e^{u}\partial_{t}u-\nabla (1+a-2e^{u})e^{u}\nabla u\\&=(1+a-2e^{u})e^{u}(\partial_{t}u-\Delta u)+4e^{2u}|\nabla u|^{2}-(1+a)e^{u}|\nabla u |^{2}\\&= (1+a-2e^{u})e^{u}(|\nabla u|^{2}+F(u))+4e^{2u}|\nabla u|^{2}-(1+a)e^{u}|\nabla u |^{2}.
\end{align*}
Combining everything we compute:
\begin{align*}
    \Box H&= 2\nabla H.\nabla u-2c\nabla F.\nabla u-2\nabla \varphi.\nabla u+2(\alpha-\beta)|\nabla \nabla u|^{2}+\Box\varphi\\& +c(1+a-2e^{u})e^{u}F(u)+(1+a)e^{u}(\alpha +2\beta)|\nabla u|^{2}+e^{2u} (2c-4\alpha-4\beta)|\nabla u|^{2}\\&+\alpha(1+a-2e^{u})e^{u}\Delta u\\[3mm]&=2\nabla H.\nabla u-2c(1+a-2e^{u})e^{u}|\nabla u|^{2}-2\nabla \varphi.\nabla u+2(\alpha-\beta)|\nabla \nabla u|^{2}+\Box\varphi\\&+c(1+a-2e^{u})e^{u}F(u) +(1+a)e^{u}(\alpha +2\beta)|\nabla u|^{2}+e^{2u} (2c-4\alpha-4\beta)|\nabla u|^{2}\\&+\alpha(1+a-2e^{u})e^{u}\Delta u\\[3mm]&=2\nabla H.\nabla u+2(\alpha-\beta)|\nabla \nabla u|^{2}-2\nabla \varphi.\nabla u+\Box\varphi +(1+a)e^{u}(\alpha +2\beta-2c)|\nabla u|^{2}\\&+e^{2u} (6c-4\alpha-4\beta)|\nabla u|^{2}+\alpha(1+a-2e^{u})e^{u}\Delta u+c(1+a-2e^{u})e^{u}F(u).
\end{align*}



\end{proof}

The proof of Theorem \ref{maintheorem} involves first restricting to a rectangle and modifying $\varphi$ by an appropriately chosen barrier function. In the following lemma, we establish some properties of the barrier function that will be needed in the proof.

\begin{lemma}[Spatial Barrier Positivity]\label{lem:spacial_pos}
Fix $\beta,\gamma>0$. There exists $b_0(\beta,\gamma)>0$ such that the following holds for all $b>b_0$. Let
\[
 R \;=\; \prod_{i=1}^{n}[p_{i},q_{i}]\subseteq \mathbb{R}^n,
 \]
be a rectangle, and define 
\begin{equation}\label{eq:phi-def} \phi_R(x)=\sum_{i=1}^{n} \left[ \frac{b}{(x_i - p_i)^2} + \frac{b}{(q_i - x_i)^2} \right].
\end{equation}
Then 
\begin{equation}\label{eq:barrier-positive}
\gamma\,\phi_R^2 - \Delta\phi_R - \frac{|\nabla \phi_R|^2}{2\gamma\beta\,\phi_R} > 0.
\end{equation}
\end{lemma}

\begin{proof}For simplicity, we omit the $R$ in $\phi_R(x)$ and simply denote it as $\phi(x)$.
Differentiating \eqref{eq:phi-def} with respect to $x_i$ gives
\[
\partial_{x_i} \phi(x) 
= -\frac{2b}{(x_i - p_i)^3} + \frac{2b}{(q_i - x_i)^3},
\]
and hence
\[
|\nabla\phi|^2 
= \sum_{i=1}^n \left[ -\frac{2b}{(x_i - p_i)^3} + \frac{2b}{(q_i - x_i)^3} \right]^2.
\]
Similarly, the second derivatives are
\[
\partial_{x_i}^2\phi(x) 
= \frac{6b}{(x_i - p_i)^4} + \frac{6b}{(q_i - x_i)^4},
\]
and therefore
\[
\Delta\phi = \sum_{i=1}^n \left[ \frac{6b}{(x_i - p_i)^4} + \frac{6b}{(q_i - x_i)^4} \right].
\]

Near a boundary face (say $x_1 \to p_1^+$), the term $\frac{b}{(x_1 - p_1)^2}$ dominates $\phi$, so
\[
\phi(x) \;\approx\; \frac{b}{(x_1 - p_1)^2}, 
\quad \phi^2 \;\approx\; \frac{b^2}{(x_1 - p_1)^4}.
\]
In this regime:
\begin{align*}
-\Delta\phi 
&\approx -\frac{6b}{(x_1 - p_1)^4}, \\
-\frac{|\nabla\phi|^2}{2\gamma\beta\,\phi} 
&\approx -\frac{\frac{4b^2}{(x_1 - p_1)^6}}{2\gamma\beta\,\frac{b}{(x_1 - p_1)^2}} 
= -\frac{2b}{\gamma\beta\,(x_1 - p_1)^4}, \\
\gamma\,\phi^2 
&\approx \frac{\gamma\,b^2}{(x_1 - p_1)^4}.
\end{align*}

Combining the three contributions near the boundary, we find:
\[
\gamma\,\phi^2 - \Delta\phi - \frac{|\nabla\phi|^2}{2\gamma\beta\,\phi}
\;\approx\; \frac{\gamma\,b^2 - 6b - \frac{2b}{\gamma\beta}}{(x_1 - p_1)^4}.
\]
Thus positivity is equivalent to
\[
\gamma b^2 - b\left(6 + \frac{2}{\gamma\beta}\right) > 0.
\]
Since $\gamma>0$, this is achieved whenever
\[
b > b_0 := \frac{6 + \frac{2}{\gamma\beta}}{\gamma}.
\]
For such $b$, \eqref{eq:barrier-positive} holds near each boundary face by symmetry of $\phi$.

Away from the boundary, $\phi$ is bounded below by a positive constant of order $b$, so $\phi^2$ is $\mathcal{O}(b^2)$ while $\frac{|\nabla\phi|^2}{\phi}$ and $\Delta\phi$ are at most $\mathcal{O}(b)$. Therefore \eqref{eq:barrier-positive} is automatic in the interior for $b$ sufficiently large.

We conclude that for all $b >b_0$, inequality \eqref{eq:barrier-positive} holds throughout $R$.
\end{proof}

\begin{remark}
Note that the choice of $b_0$ does not depend on the rectangle $R$. This allows one to pass to the limit as $R$ exhausts $\mathbb{R}^n$, retaining the above inequality.
\end{remark}

The following lemma establishes the properties of $\varphi$ needed in Section \ref{sec:main_proof}; closely related constructions appear in \cite{cck15, Cao_2017, bbc19}.

\begin{lemma}\label{lem:computations_2}
Let $\mu>\nu> 0$, $\eta> 0$, and set
\[
\gamma=\mu^{2}-\nu^{2}>0.
\]

Define
\[
\xi(t)
=
\frac{\eta}{1-e^{2\mu\eta t}}
\left(
\frac{e^{2\mu\eta t}}{\nu-\mu}
-\frac{1}{\nu+\mu}
\right).
\]
Then $\xi(t)>0$ for all $t>0$,
\[
\lim_{t\to0^{+}}\xi(t)=+\infty,
\]
and
\[
\xi_t+(\mu\xi)^2-(\eta+\nu\xi)^2=0.
\]

Moreover, suppose that $A>0$ satisfies
\[
\gamma A=\eta\nu,
\]
and define
\[
\varphi(t)=\xi(t)+A.
\]
Then
\[
\lim_{t\to0^{+}}\varphi(t)=+\infty,
\]
and
\[
\varphi_t+(\mu\varphi)^2-(\eta+\nu\varphi)^2>0
\]
for every $t>0$.

\end{lemma}

\begin{proof}
Since $\eta>0$ and
\[
\gamma=\mu^{2}-\nu^{2}=(\mu-\nu)(\mu+\nu),
\]
we may rewrite $\xi$ as
\begin{align*}
\xi(t)
&=
\frac{\eta}{1-e^{2\mu\eta t}}
\left(
\frac{e^{2\mu\eta t}}{\nu-\mu}
-\frac{1}{\nu+\mu}
\right)\\
&=
\frac{\eta}{\mu^{2}-\nu^{2}}
\left(
\mu\frac{e^{2\mu\eta t}+1}{e^{2\mu\eta t}-1}
+\nu
\right)\\
&=
\frac{\eta}{\gamma}
\left(
\mu\coth(\mu\eta t)+\nu
\right).
\end{align*}
Because $\coth(\mu\eta t)>1$ for $t>0$, it follows that
\[
\xi(t)>0.
\]
Moreover, since $\coth z\to+\infty$ as $z\to0^{+}$,
\[
\lim_{t\to0^{+}}\xi(t)=+\infty.
\]

Differentiating gives
\[
\xi_t
=
\frac{\mu^{2}\eta^{2}}{\gamma}
(1-\operatorname{coth}^2(\mu\eta t)).
\]
On the other hand,
\[
\mu\xi
=
\frac{\mu\eta}{\gamma}
\left(
\mu\coth(\mu\eta t)+\nu
\right),
\]
while
\begin{align*}
\eta+\nu\xi
&=
\eta+
\frac{\nu\eta}{\gamma}
\left(
\mu\coth(\mu\eta t)+\nu
\right)\\
&= 
\frac{\mu\eta}{\gamma}
\left(
\mu+\nu\coth(\mu\eta t)
\right).
\end{align*}
Consequently,
\begin{align*}
(\mu\xi)^2-(\eta+\nu\xi)^2
&=
\frac{\mu^{2}\eta^{2}}{\gamma^{2}}
\left[
\left(\mu\coth(\mu\eta t)+\nu\right)^2
-
\left(\mu+\nu\coth(\mu\eta t)\right)^2
\right]\\
&=
\frac{\mu^{2}\eta^{2}}{\gamma^{2}}
(\mu^{2}-\nu^{2})
\left(\coth^{2}(\mu\eta t)-1\right)\\
&=
\frac{\mu^{2}\eta^{2}}{\gamma}
(\operatorname{coth}^2(\mu\eta t)-1).
\end{align*}
Therefore
\[
\xi_t+(\mu\xi)^2-(\eta+\nu\xi)^2=0.
\]

Now define
\[
\varphi=\xi+A.
\]
Using the equation satisfied by $\xi$, we obtain
\begin{align*}
&\varphi_t+(\mu\varphi)^2-(\eta+\nu\varphi)^2\\
&=
\xi_t+\mu^{2}(\xi+A)^2
-\bigl(\eta+\nu(\xi+A)\bigr)^2\\
&=
\left[
\xi_t+\mu^{2}\xi^2-(\eta+\nu\xi)^2
\right]
+
A\left(
2(\mu^{2}-\nu^{2})\xi
+(\mu^{2}-\nu^{2})A
-2\eta\nu
\right)\\
&=
A\left(2\gamma\xi+\gamma A-2\eta\nu\right).
\end{align*}
Using
\[
\gamma A=\eta\nu,
\]
this becomes
\[
\varphi_t+(\mu\varphi)^2-(\eta+\nu\varphi)^2
=
A(2\gamma\xi-\eta\nu).
\]
From the explicit expression for $\xi$,
\[
\gamma\xi
=
\eta\left(\mu\coth(\mu\eta t)+\nu\right),
\]
and hence
\begin{align*}
2\gamma\xi-\eta\nu
&=
\eta\left(
2\mu\coth(\mu\eta t)+\nu
\right)>0.
\end{align*}
Thus
\[
\varphi_t+(\mu\varphi)^2-(\eta+\nu\varphi)^2>0.
\]
The blow-up of $\varphi$ at $t=0$ follows immediately from that of $\xi$. This completes the proof.
\end{proof}
\section{Proof of the Main Theorem}\label{sec:main_proof}
In this section, we present the proof of Theorem \ref{maintheorem}.
\begin{proof}

Let $R_k =\{x\in \mathbb{R}^n \,\vert \, x_i\in [-k,k],\,1\leq i\leq n\}$ and consider the modified Harnack quantity $$ H_k =H+ \phi_k$$ where $$\phi _k =\sum_{i=1}^n \left( \frac{b}{(x_i-k)^2}+ \frac{b}{(x_i+k)^2}\right) $$
for some $b>0$. Since $H_k\xrightarrow[k\to\infty]{C^0_{loc}}H $, it suffices to show that $H_k>0$. Therefore, we assume for the sake of contradiction that there is a point $(x_0,t_0)\in R_k \times (0, \infty)$ for which $H_k(x_0,t_0)=0$. By \cite{cck15}, we have $\lim_{t\searrow 0 }\varphi(t)=\infty, $ and so $\liminf_{t\searrow 0}H_k>0$. Moreover, since $H_k$ blows up at the boundary of $R_k$, it achieves its minimum in the interior of  $R_k$. Therefore, we may assume that $(x_0,t_0)$ is a local minimum and is minimal in time among all points $(x,t)$ for which $H_k(x,t)=0$. It follows that
\begin{align}
    \Box (H_k)(x_0,t_0) & \leq 0 \label{eq:box_H_1} \\ \nabla H_k(x_0,t_0)&=0 .\label{eq:gradient_H}
\end{align}
Set $\varphi_k= \varphi+ \phi_k$, so that $\varphi_k \xrightarrow[k\to\infty]{C^0_{loc}} \varphi$. From Lemma \ref{lem:computations}, we obtain that 
\begin{align*}
    \Box H_k\geq & 2\nabla u \cdot \nabla H_k + 2\frac{(\alpha-\beta)}{n\alpha^{2}} (H_k-\beta|\nabla u|^{2}-cF-\varphi_k)^{2}+ \Box\varphi_k-2\nabla u \cdot \nabla \varphi_k 
\\
& + (1+a - 2e^u) e^u \cdot (H_k-\beta |\nabla u|^{2}-cF-\varphi_k) + e^u \cdot |\nabla u|^2 (1 + a)(\alpha + 2 \beta - 2c) \nonumber \\
& + e^u \cdot |\nabla u|^2 (- 2 e^u)(2 \alpha + 2 \beta - 3c) + c(1+a-2e^u) e^u (1-e^u)(e^u - a) \nonumber. 
\end{align*}
Set $\gamma=\frac{2(\alpha-\beta)}{n\alpha^{2}}$, so that $\gamma>0$ by \eqref{iden:main_thm_1}. Then
\begin{align}
   ( \Box H_k)(x_0,t_0)&\geq \gamma (\beta|\nabla u|^{2}+cF+\varphi_k)^{2}-(1+a-2e^{u})e^{u}(\beta|\nabla u|^{2}+\varphi_k)+\Box \varphi_k-2\nabla u.\nabla\varphi_k \nonumber \\&+e^u \cdot |\nabla u|^2 (1 + a)(\alpha + 2 \beta - 2c)+e^u \cdot |\nabla u|^2 (- 2 e^u)(2 \alpha + 2 \beta - 3c) \nonumber\\&= \gamma(\varphi_k+cF)^{2}+(2e^{2u}-(1+a)e^{u})\varphi_k+\beta^{2}\gamma|\nabla u|^{4}+\Box \varphi_k-2\nabla u.\nabla\varphi_k \nonumber\\&+|\nabla u|^{2}(2\gamma\beta \varphi_k +2\gamma\beta c F-\beta (1+a-2e^{u})e^{u}+e^{u}(1+a)(\alpha+2\beta-2c)\nonumber \\ 
  & \hspace{15mm}-2e^{2u}(2\alpha+2\beta-3c))\nonumber\\&=\gamma\varphi_k^{2}+(2(1-c\gamma)e^{2u}+(1+a)(2c\gamma-1)e^{u}-2c\gamma a)\varphi_k+\gamma c^{2}F^{2}+\beta^{2}\gamma|\nabla u|^{4}\\& +\Box \varphi_k-2\nabla u.\nabla\varphi_k+|\nabla u |^{2}(-2\gamma\beta ca+2\gamma\beta \varphi_k+e^{2u}(6c-4\alpha-2\beta(1+c\gamma)) \nonumber \\& \hspace{55mm}+ e^{u}(1+a)(\beta(1+2\gamma c)+\alpha-2c))  \nonumber.
\end{align}

Define 
\begin{align*}
    \omega&=6c-4\alpha-2\beta(1+c\gamma), \\\delta&=(1+a)(\beta(1+2\gamma c)+\alpha-2c),
\end{align*}
so that $(1-c\gamma),\omega>0$ by \eqref{iden:main_thm_1} and \eqref{iden:main_thm_2}. Completing the square we have
\begin{align*}
    2(1-c\gamma)e^{2u}+(1+a)(2c\gamma-1)e^{u}&\geq -\frac{(2c\gamma-1)^{2}(1+a)^{2}}{8(1-c\gamma)},\\\omega e^{2u}+\delta e^{u}&\geq -\frac{\delta^{2}}{4\omega },\\2\gamma\beta\varphi_k|\nabla u|^{2}-2\nabla u.\nabla\varphi_k &\geq -\frac{|\nabla\varphi_k|^{2}}{2\gamma\beta\varphi_k},
\end{align*}
which gives
\begin{align*}
(\Box H_k)(x_0,t_0)\geq \Box\varphi_k+\gamma\varphi_k^{2}-\frac{(2c\gamma-1)^{2}(1+a)^{2}}{8(1-c\gamma)}\varphi_k+\beta^{2}\gamma|\nabla u|^{4}+|\nabla u|^{2}(-2\gamma\beta a c -\frac{\delta^{2}}{4\omega}) -\frac{|\nabla \varphi_k|^{2}}{2\gamma\beta\varphi_k}.
\end{align*}
Completing the square once again we have
\begin{align*}
    \beta^{2}\gamma|\nabla u|^{4}+|\nabla u|^{2}(-2\gamma\beta ac  -\frac{\delta^{2}}{4\omega})\geq -\frac{(-\gamma\beta ac-\frac{\delta^{2}}{8\omega})^{2}}{\gamma\beta^{2}},
\end{align*}
and so
\begin{align}
(\Box H_k)(x_0,t_0)&\geq \Box\varphi_k+\gamma\varphi_k^{2}-\frac{(2c\gamma-1)^{2}(1+a)^{2}}{8(1-c\gamma)}\varphi_k-\frac{1}{\gamma\beta^{2}}(\gamma\beta ac+\frac{\delta^{2}}{8\omega})^{2}-\frac{|\nabla \varphi_k|^{2}}{2\gamma\beta\varphi_k} \nonumber \\&\geq (\varphi_k)_{t}+ \gamma\varphi_k^{2}-\frac{(2c\gamma-1)^{2}(1+a)^{2}}{8(1-c\gamma)}\varphi_k- \frac{(2c\gamma-1)^{2}(1+a)^{2}}{8(1-c\gamma)}A-\frac{1}{\gamma\beta^{2}}(\gamma\beta ac+\frac{\delta^{2}}{8\omega})^{2} \nonumber\\&+\gamma\phi_k^{2} +(2A\gamma- \frac{(2c\gamma-1)^{2}(1+a)^{2}}{8(1-c\gamma)})\phi_k -\Delta \phi_k-\frac{|\nabla\phi_k|^{2}}{2\gamma\beta\phi_k} \nonumber\\&= (\varphi_k)_{t}+(\mu \varphi_k)^{2}-(\eta+\nu\varphi_k)^{2}+\gamma\phi_k^{2}-\Delta\phi_k-\frac{|\nabla \phi_k|^{2}}{2\gamma\beta\phi_k},
\end{align}
where we set
\begin{align*}
   A&= \frac{(2c\gamma-1)^{2}(1+a)^{2}}{16\gamma(1-c\gamma)},\\ \eta^{2}&=\frac{1}{\gamma\beta^{2}}(\gamma\beta ac+\frac{\delta^{2}}{8\omega})^{2}+\frac{(2c\gamma-1)^{2}(1+a)^{2}}{8(1-c\gamma)}A,\\\eta\nu&=\frac{(2c\gamma-1)^{2}(1+a)^{2}}{16(1-c\gamma)},\\ \mu^{2}&=\nu^{2}+\gamma.
\end{align*}
Note that $A,\eta,\nu,$ and $\mu$, satisfy the hypothesis of Lemma \ref{lem:computations_2}. As in Lemma \ref{lem:spacial_pos}, we may choose $b>0$ large (independent of the rectangle $R_k$) so that
\begin{align}
    \gamma\phi^{2}_k-\Delta \phi_k -\frac{|\nabla\phi_k|^{2}}{2\gamma\beta\phi_k}> 0
\end{align}
 on $R_k$. Combining the above we have
\begin{align*}
(\Box H_k)(x_{0},t_{0})> 0,
\end{align*}
which is a contradiction. 

\end{proof}
\section{Applications}\label{sec:applications}

In this section we use Theorem \ref{maintheorem} to prove Theorem \ref{thm:app} and Theorem \ref{thm:app_2}. 

\subsection{Classical Harnack}
First, we use our differential Harnack estimate to prove Theorem \ref{thm:app}, which gives a classical Harnack inequality comparing values of positive solutions at different points.

\begin{proof}
 Let $\Gamma(t)$ be any space-time line connecting $(x_{1},t_{1})$ and $(x_{2},t_{2})$ and as before define $u=\operatorname{log} f$.  By Theorem \ref{maintheorem} we have
 \begin{align*}
     u(x_{2},t_{2})-u(x_{1},t_{1})&=\int_{t_{1}}^{t_{2}}u_{t}+\nabla u \cdot\Gamma^{\prime}\,dt\\ &= \int_{t_{1}}^{t_{2}} \Delta u +|\nabla u|^{2}+F(u)  +\nabla u \cdot \Gamma^{\prime}\,dt\\
     &\geq \int_{t_{1}}^{t_{2}}-\frac{1}{\alpha}(\beta|\nabla u |^{2}+cF(u)+\varphi(t))+|\nabla u|^{2}+F(u)+\nabla u \cdot \Gamma^{\prime}\,dt\\&\geq -\frac{1}{4(1-\frac{\beta}{\alpha})}\int_{t_{1}}^{t_{2}}|\Gamma^{\prime}|^{2}dt+(1-\frac{c}{\alpha})\int_{t_{1}}^{t_{2}}F(u)\,dt \\&-\frac{\eta}{\alpha}\int_{t_{1}}^{t_{2}}\big[\frac{1}{\nu-\mu}\frac{e^{2\mu\eta t}}{1-e^{2\mu\eta t}}-\frac{1}{\nu+\mu}\frac{1}{1-e^{2\mu\eta t}}\big]\, dt-\frac{A(t_{2}-t_{1})}{\alpha}\\ 
     &= -\frac{1}{4(1-\frac{\beta}{\alpha})}\int_{t_{1}}^{t_{2}}|\Gamma^{\prime}|^{2}dt+(1-\frac{c}{\alpha})\int_{t_{1}}^{t_{2}}F(u)\,dt -\frac{A(t_{2}-t_{1})}{\alpha}\\&+\frac{\eta}{\alpha}\big[-\frac{1}{2\mu\eta(\nu+\mu)}\text{log}(1-e^{-2\mu\eta t})+\frac{1}{2\mu\eta(\nu-\mu)}\text{log}(e^{2\mu\eta t}-1)\big]_{t_{1}}^{t_{2}}\\ &=  -\frac{1}{4(1-\frac{\beta}{\alpha})}\frac{|x_{2}-x_{1}|^{2}}{t_{2}-t_{1}}+(1-\frac{c}{\alpha})\int_{t_{1}}^{t_{2}}F(u)\,dt-\frac{A(t_{2}-t_{1})}{\alpha}\\&+\frac{\eta}{\alpha}\big[-\frac{1}{2\mu\eta(\nu+\mu)}\text{log}(1-e^{-2\mu\eta t})+\frac{1}{2\mu\eta(\nu-\mu)}\text{log}(e^{2\mu\eta t}-1)\big]_{t_{1}}^{t_{2}}. 
 \end{align*}
Since $0<f\le1$, we have
\[
0\le F(u)=(1-f)(f-a)\le \frac{(1-a)^2}{4}.
\]
Therefore,
\[
\left(1-\frac{c}{\alpha}\right)F(u)
\ge
-\frac14\left(\frac{c}{\alpha}-1\right)_{+}(1-a)^2,
\]
where $r_+=\max\{r,0\}$. Hence
\[
\left(1-\frac{c}{\alpha}\right)
\int_{t_1}^{t_2}F(u)\,dt
\ge
-\frac{t_2-t_1}{4}
\left(\frac{c}{\alpha}-1\right)_{+}(1-a)^2.
\]
Substituting this estimate into the previous inequality yields
\[
\begin{aligned}
u(x_2,t_2)-u(x_1,t_1)
\ge\;&
-\frac{1}{4\left(1-\frac{\beta}{\alpha}\right)}
\frac{|x_2-x_1|^2}{t_2-t_1}
-\frac{t_2-t_1}{4}
\left(\frac{c}{\alpha}-1\right)_{+}(1-a)^2
-\frac{A(t_2-t_1)}{\alpha}
\\
&
+\frac{\eta}{\alpha}
\left[
-\frac{1}{2\mu\eta(\nu+\mu)}
\log\!\left(1-e^{-2\mu\eta t}\right)
+\frac{1}{2\mu\eta(\nu-\mu)}
\log\!\left(e^{2\mu\eta t}-1\right)
\right]_{t_1}^{t_2}.
\end{aligned}
\]
Exponentiating both sides gives the desired estimate.
 
If $f\geq 1$ and $c\geq \alpha$, then
\[
F(u)=(1-f)(f-a)\leq 0
\]
and
\[
1-\frac{c}{\alpha}\leq 0.
\]
Consequently,
\[
\left(1-\frac{c}{\alpha}\right)F(u)\geq 0.
\]
Returning to the preceding integral estimate and discarding this
nonnegative term, we obtain
\[
\begin{aligned}
u(x_2,t_2)-u(x_1,t_1)
\geq {}&
-\frac{1}{4\left(1-\frac{\beta}{\alpha}\right)}
\frac{|x_2-x_1|^2}{t_2-t_1}
-\frac{A(t_2-t_1)}{\alpha}
\\
&+
\frac{\eta}{\alpha}
\left[
-\frac{1}{2\mu\eta(\nu+\mu)}
\log\left(1-e^{-2\mu\eta t}\right)
+
\frac{1}{2\mu\eta(\nu-\mu)}
\log\left(e^{2\mu\eta t}-1\right)
\right]_{t_1}^{t_2}.
\end{aligned}
\]
Exponentiating gives the second assertion of the theorem.
\end{proof}

\subsection{Traveling wave solutions}

Finally, we prove Theorem \ref{thm:app_2}, which gives a bound on the wave speed of traveling wave solutions.

\begin{proof}
Recall that our
differential Harnack estimate gives
\[
\alpha\Delta u+\beta |\nabla u|^2+cF(u)+\varphi(t)\geq 0,
\]
where
\[
u=\log f,\qquad F(u)=(1-e^u)(e^u-a).
\]

We first rewrite the Harnack quantity in terms of $f$. Since
\[
\Delta u=\frac{\Delta f}{f}-\frac{|\nabla f|^2}{f^2},
\]
and from the equation
\[
f_t=\Delta f+fF(u),
\]
we have
\[
\Delta f=f_t-fF(u).
\]
Therefore
\[
\Delta u
=
\frac{f_t}{f}-F(u)-\frac{|\nabla f|^2}{f^2}.
\]
Substituting this into the differential Harnack inequality gives
\begin{equation}
\alpha\frac{f_t}{f}
+(\beta-\alpha)\frac{|\nabla f|^2}{f^2}
+(c-\alpha)F(u)
+\varphi(t)\geq0.
\label{eq:Hf}
\end{equation}

For a traveling wave solution we have
\[
f_t=s\,v_{x_n}.
\]
Thus \eqref{eq:Hf} becomes
\[
\alpha s\frac{v_{x_n}}{v}
-(\alpha-\beta)\frac{|\nabla v|^2}{v^2}
+(c-\alpha)(1-v)(v-a)
+\varphi(t)\geq0.
\]

Using
\[
|v_{x_n}|\leq |\nabla v|,
\]
we obtain
\[
0\leq
\alpha |s|\frac{|\nabla v|}{v}
-(\alpha-\beta)\frac{|\nabla v|^2}{v^2}
+(c-\alpha)(1-v)(v-a)
+\varphi(t).
\]

Set
\[
X=\frac{|\nabla v|}{v}.
\]
Completing the square gives
\[
-(\alpha-\beta)X^2+\alpha |s|X
\leq
\frac{\alpha^2s^2}{4(\alpha-\beta)}.
\]
Therefore
\[
0\leq
\frac{\alpha^2s^2}{4(\alpha-\beta)}
+(c-\alpha)(1-v)(v-a)
+\varphi(t).
\]

Taking the limit as $t\to \infty$ gives
\[
0\leq
\frac{\alpha^2s^2}{4(\alpha-\beta)}
-(c-\alpha)a+\varphi_\infty .
\]


Hence we obtain the following lower bound for the wave speed:
\[
s^2\geq
\frac{4(\alpha-\beta)}{\alpha^2}
\left((c-\alpha)a-\varphi_\infty\right).
\]

In particular, whenever
\[
(c-\alpha)a>\varphi_\infty,
\]
the differential Harnack estimate gives a nontrivial lower bound on
the speed of any positive traveling wave solution.
\end{proof}

\paragraph{\textbf{Acknowledgements:}}  S. Teng's research was partially supported by Cornell University's Summer Program for Undergraduate Research. 
\bibliographystyle{plain}
\bibliography{bio}

\end{document}